\documentclass[a4paper,oneside,12pt]{amsart}

\reversemarginpar

\usepackage{amssymb}

\usepackage{amssymb,latexsym,amsmath,amsfonts,amscd}
\usepackage{eucal}
\usepackage{color}
\usepackage{mathrsfs}
\usepackage{dsfont}

\DeclareSymbolFont{SY}{U}{psy}{m}{n}
\DeclareMathSymbol{\emptyset}{\mathord}{SY}{'306}

\theoremstyle{plain}

\newtheorem*{lemA}{Lemma A}
\newtheorem{thm}{Theorem}[section]
\newtheorem*{thmAFV}{Theorem AFV}
\newtheorem*{thmnonumber}{Theorem}
\newtheorem{cor}[thm]{Corollary}
\newtheorem{lem}[thm]{Lemma}
\newtheorem{prop}[thm]{Proposition}
\theoremstyle{definition}
\newtheorem{defn}[thm]{Definition}
\newtheorem{rem}[thm]{Remark}
\newtheorem{ex}[thm]{Example}

\numberwithin{equation}{section}

\def\C{{\mathbb C}}

\def\R{\mathbb{R}}

\def\MM{M}

\def\Re{\operatorname{Re}}

\def\Hol{\operatorname{Hol}}

\def\beq{\begin{eqnarray}}
\def\eeq{\end{eqnarray}}
\def\beqa{\begin{eqnarray*}}
\def\eeqa{\end{eqnarray*}}
\def\wt{\widetilde}

\def\Om{\Omega}
\def\ovr{\overline}
\def\dist{\operatorname{dist}}
\def\Span{\operatorname{Span}}
\def\Ran{\operatorname{Ran}}
\def\clos{\operatorname{clos}}
\def\sm{\setminus}
\def\pt{\partial}
\def\eps{\epsilon}
\def\de{\delta}
\def\ga{\gamma}
\def\beqn{\begin{equation}}
\def\eeqn{\end{equation}}
\def\Mult{\operatorname{Mult}(\wt\cH)}
\def\BC{{\mathbb C}}
\def\cH{{\mathcal H}}
\def\cN{{\mathcal N}}

\def\mg#1{}

\def\la{\lambda}
\def\si{\sigma}
\def\eps{\varepsilon}

\def\Lip{\operatorname{Lip}}
\def\Ran{\operatorname{Ran}}
\def\ind{\operatorname{ind}}
\renewcommand{\epsilon}{\varepsilon}
\renewcommand{\phi}{\varphi}

\begin{document}
\title{Infinite-dimensional features of matrices and pseudospectra}
\author{Avijit Pal, Dmitry V. Yakubovich}
\vskip-1cm
\address[A. Pal]{Department of Mathematics and Statistics, Indian Institute of Science, Education and
Research Kolkata, Mohanpur - 741 246, India
 }
\email{avijitmath@gmail.com}
\address[D. V. Yakubovich]{
Departamento de Matem\'aticas, Universidad Aut\'onoma de Madrid, Cantoblanco, 28049 Madrid, Spain
\newline and Instituto de Ciencias Matem\'aticas (CSIC-UAM-UC3M-UCM), Madrid, Spain.}
\email {dmitry.yakubovich@uam.es}

\subjclass[2010]{47B35 (primary)}

\keywords{Pseudospectra, Quasitriangular,
Cowen-Douglas class, nonnormal matrices,
approximation of spectra, finite section method.}
\begin{abstract}
Given a Hilbert space operator $T$, the level sets of function
$\Psi_T(z)=\|(T-z)^{-1}\|^{-1}$ determine the so-called
pseudospectra of $T$. We set $\Psi_T$ to be zero on the spectrum
of $T$. After giving some elementary properties of $\Psi_T$
(which, as it seems, were not noticed before), we apply them to
the study of the approximation. We prove that for any operator
$T$, there is a sequence $\{T_n\}$ of finite matrices such that
$\Psi_{T_n}(z)$ tends to $\Psi_{T}(z)$ uniformly on $\C$. In this
proof, quasitriangular operators play a special role. This is
merely an existence result, we do not give a  concrete construction
of this sequence of matrices.

One of our main points is to show how to use infinite-dimensional
operator models in order to produce examples and counterexamples
in the set of finite matrices of large
size. In particular, we
get a result, which means, in a sense, that
the pseudospectrum of a nilpotent matrix
can be anything one can imagine. We also study the norms of the
multipliers in the context of Cowen--Douglas class operators. We use these results to show that, to the
opposite to the function $\Psi_{S}$,
the function $\|\sqrt{S-z}\,\|$
for certain finite matrices $S$ may oscillate
arbitrarily fast
even far away from the spectrum.
\end{abstract}
\maketitle
\vskip-.5cm

\section{Introduction}
Let $\mathcal H$ be a complex separable Hilbert space and
$\mathcal B(\mathcal H)$ be the algebra of bounded operators on
$\mathcal H$, equipped with the supremum norm. Given an operator
$T\in \mathcal B(\mathcal H)$, put
\[
\Psi_{T}(z) = \left\{ \begin{array}{ll}
        0  & \mbox{if $z \in \sigma(T) $};\\
        \|(T-z)^{-1}\|^{-1} & \mbox{if $z \not \in \sigma(T)$}.\end{array} \right.
\]
This function is closely related with so-called
$\epsilon$-pseudospectra of  $T$, defined by
\[
\sigma_{\epsilon}(T)=\{z \in \mathbb C:\Psi_{T}(z)< \epsilon\}
\]
(here $\eps>0$).
It is well known that
\[
\sigma_{\epsilon}(T)=\bigcup_{\|A\|<\epsilon}\sigma(T+A),
\]
see, for instance \cite{{chartin}, {Gall}, {sharg-defn}}.
While the $\epsilon$-pseudospectrum of a normal
operator in a Hilbert space coincides with the
$\epsilon$-neighbourhood of the spectrum, the situation is more
involved for non-normal operators.
It is well-known that the spectral properties of a nonnormal operator (or matrix)
not only depend on its spectrum, but are also influenced by the resolvent growth.
The pseudospectra are a good language to describe this growth,
and their importance
has been widely recognized in the recent years.
Their applications include the finite section method for Toeplitz matrices,
growth bounds for semigroups, numerics for differential operators, matrix iterations,
linear models for turbulence, etc.
We refer to the book \cite{Trefeth-book} by Trefethen and Embree and to
the Trefethen's review \cite{Trefeth-SIAM} for comprehensive accounts.
Much effort has been devoted to the calculation of pseudospectra of matrices
\cite{Trefeth-comp-ps}.

By a \textit{filtration} on $\mathcal H$, we mean a sequence
$\{P_n\}$ of finite rank orthogonal projections such that $\Ran
P_n \subseteq \Ran P_{n+1}$ and $\cup_n \Ran P_n$ is dense in
$\mathcal H$. The corresponding sequence of finite dimensional
operators $T_n = P_n T_{\mid \Ran P_n}$ will be referred to as
\textit{finite sections of~$T$}.

Recall that an operator
$T \in \mathcal B(\mathcal H)$ is said to be
\textit{quasitriangular} if there is
a filtration $\{P_n\}$  such that
$\lim_{n \rightarrow
\infty}\|(I-P_n)TP_n\|=0$.
If there is a
filtration $\{P_n\}$ such that both
$\lim_{n \rightarrow
\infty}\|(I-P_n)TP_n\|=0$ and
$\lim_{n \rightarrow
\infty}\|P_nT(I-P_n)\|=0$, then $T$ is said to be quasidiagonal.
In these cases, we will refer to
$\{P_n\}$ as to a \textit{filtration, corresponding to a quasitriangular (quasidiagonal)
operator $T$}.

It is well known that spectra do not
necessarily behave well under limiting procedures, even for
a sequence of bounded operators on some Hilbert space
$\mathcal H$ converging in operator norm.
For example, consider the bilateral weighted shift on
$\ell_2(\mathbb Z)$, defined by $T(s) e_j=e_{j+1}$ for $j\ne0$ and
$T(s) e_0=s e_1$ (here $\{e_j: j\in\mathbb{Z}\}$ is the standard
basis of $\ell_2(\mathbb Z)$). Then the spectrum of $T(s)$ equals
to the unit circle for any nonzero $s$, while the spectrum of
$T(0)$ is the whole closed unit disc, and there is no convergence
of spectra as $s\to 0$.
For the case of pseudospectra, the
situation is better. It was noticed by many authors that, to the
opposite to usual spectra, pseudospectra supply a vast
quantitative information on the behavior of powers of operators,
the semigroups they generate, etc. Our work also gives some
results in this direction.

\smallskip
Our main results are as follows.
In Section~\ref{elem-estims}, we prove several elementary
estimates and properties for the function  $\Psi_T(z)$.  In
particular, we show that
it is locally semiconvex (see the definition below).
The list of
these properties certainly can be extended. However, the question
of describing all functions on $\BC$ representable as $\Psi_T(z)$
for a Hilbert (or Banach) space operator $T$ seems to be open and
might be interesting. We use the results of
Section~\ref{elem-estims} in the next sections. We believe that
these results may also be important for algorithms of numerical
calculation of pseudospectra.

Section~\ref{gen-conv-thms} is devoted to general convergence
results for pseudospectra and for the function $\Psi_T(z)$.  One
of our starting points was the result by  N. Brown, which says
that if $T$ is quasidiagonal operator and $\{P_n\}$ is a
corresponding filtration, then for any $\eps$, the
$\eps$-pseudospectra of $T_n$ tend to the $\eps$-pseudospectrum of
$T$, see \cite{brown}, Theorem~3.5~(1). We observe that a similar
assertion holds also for quasitriangular operators.
 We prove that for any
quasitriangular operator $T$ and the corresponding filtration $\{P_n\}$,
the functions $\Psi_{T_n}$ tend uniformly to $\Psi_{T}$ on the whole complex plane.
This permits us to show that for an arbitrary operator $T$, there is a sequence
of matrices $S_n$ such that $\Psi_{S_n}$ tend uniformly to $\Psi_{T}$ on $\mathbb C$.
Here we use the theorem by Apostol, Foia\c{s} and Voiculescu, which
characterizes quasitriangular operators
in terms of semi-Fredholmness.

In Section~\ref{sect-shapes}, we use the above convergence results
to prove that, in a sense, the function  $\Psi_T(z)$, corresponding to a nilpotent
matrix $T$, can have any imaginable shape. In this proof, we apply our approximation
results to the adjoint to the operator of multiplication by the independent variable on
the Hardy space $H^2$ of
a domain in $\BC$ and to direct sums of such operators.

The function $\Psi_T(z)$ only depends on the norms of the
resolvent of $T$. One can ask about estimates of other functions
of $T$. In Section~\ref{sect-multipls}, we prove an approximation
result in this direction. We show that for a Cowen-Douglas class
operator $T$, a function $f$, holomorphic at $0$, and a filtration
$\{P_n\}$, chosen in a special way (so that all finite sections
$T_n$ are nilpotent), the norms of $f(T_n)$ are uniformly bounded
if and only if $f$ belongs to a certain multiplier space. Notice
that the Cowen-Douglas class is a particular (and well-understood)
subclass of quasitriangular operators.

This result motivated Example~\ref{ex-mult}, where we show that, to
the opposite to the function $\Psi_S(z)$
(which is Lipschitz with constant $1$),
no uniform Lipschitz estimates for the function
$\|\sqrt{S-z}\,\|$ are possible in a neighbourhood of $1$, even if
$S$ is assumed to be a finite nilpotent matrix.

\

\section{Elementary estimates}
\label{elem-estims}
Let $T$ be an operator on a Hilbert space $\mathcal{H}$.
We denote the
kernel of $T$ and range of $T$ by $\ker T$ and $\Ran T$
respectively. If $\mathcal H_{0}$ is a closed subspace of
$\mathcal H$, then we shall write $\mathcal H_{0}\subseteq
\mathcal H$. For  $T \in \mathcal B(\mathcal H)$, we shall denote
the spectrum, point spectrum, left spectrum and right spectrum by
$\sigma(T), \sigma_{p}(T), \sigma_{l}(T)$ and $\sigma_{r}(T)$
respectively.

Given a point $z\in \C$,
we
recall that \textit{the injectivity radius} $j_{T}(z)$
and
\textit{the surjectivity radius} $k_{T}(z)$ of $T-z$
are defined by
\begin{align*}
j_{T}(z)& =\inf\{\|(T-z)h\|: h\in \mathcal H, \|h\|=1\}, \\
k_{T}(z)& =\sup\{r:(T-z)B_{\mathcal H}\supset r B_{\mathcal H}\},
\end{align*}
where $B_{\mathcal H}=\{h\in \mathcal H: \|h\|\leq1\}$. The
following proposition gives a relation between these two
characteristics.

\begin{prop}\mbox{\cite[Theorem 7, Theorem 8]{ muller}:}
\label{Muller}
\begin{itemize}

\item[(i)]
For any $T \in \mathcal B(\mathcal H)$ and any $z\in\C$,
$j_{T}(z)=k_{T^*}(\bar{z})$.
\item[(ii)]
If $T -z$ is invertible, then
$$
j_{T}(z)=k_{T}(z)=\Psi_T(z).
$$
\end{itemize}
\end{prop}

As a consequence, we will prove the following lemma.
\begin{lem}\label{psi}
The following assertions hold.
\begin{enumerate}
\item $\Psi_{T}(z)=\min \big(j_{T}(z), j_{T^*}(\bar{z})\big)$.
\item
If $j_{T}(z)>0$ and $j_{T^*}(\bar{z})>0$, then $j_{T}(z)=j_{T^*}(\bar{z})=\Psi_{T}(z)$.

\end{enumerate}
\end{lem}
\begin{proof}
Suppose first that both $j_{T}(z)>0$ and $j_{T^*}(\bar z)>0$.  Then using Proposition \ref{Muller}, we get
$j_{T}(z)=j_{T^*}(\bar{z})=\Psi_{T}(z)$, so that (1) holds in this case. This also gives (2).

Now suppose $j_{T}(z)=0$. Then $T-z$ is not invertible.
Hence $\Psi_{T}(z)=0=\min \big(j_{T}(z), j_{T^*}(\bar{z})\big)$.
Similarly, if $j_{T^*}(\bar{z})=0$, then also $\Psi_{T}(z)=0$. This completes the proof.
\end{proof}

\medskip

Let $f: K\to \mathbb C$ be a function, defined on a
subset $K$ of the complex plane
and let $C>0$. In what follows, we will write $f\in
\Lip_C(K)$ if $f$ is a Lipschitz function with constant $C$, that
is, $|f(z)-f(z')|\le C |z-z'|$ for all $z, z'\in K$.

\begin{lem}\label{lips1}
For any $ T \in \mathcal B(\mathcal H)$, $j_T\in Lip_1(\C)$.
\end{lem}
\begin{proof}
Take any $z,z^{\prime} \in \mathbb C $. Then we have
$\|(T-z^{\prime}) h \|\leq \|(T-z) h \|+ |z-z^{\prime}|$ for
any $h \in \mathcal H$ with $\|h\|=1$. Therefore
$j_{T}(z')-j_{T}(z)\leq |z-z^{\prime}|$. By symmetry, this implies
the statement of Lemma.
\end{proof}

Since $\Psi_{T}(z)=\min \big(j_{T}(z), j_{T^*}(\bar{z})\big)$ and
the minimum of two $\Lip_1(\C)$ functions is again a $\Lip_1(\C)$
function, we get the following corollary.
\begin{cor}\label{lips3}
For any $ T \in \mathcal B(\mathcal H)$, $\Psi_T\in \Lip_1(\C)$.
\end{cor}
This fact is known, see Theorem 9.2.15 from
the
E. Brian Davies' book \cite{Dav}. It holds, in fact, for any
Banach space operator.

Put
\[
\rho_{\theta}(T)=\sup_{\|h\|=1} \Re \, \langle e^{-i\theta} T h, h \rangle, \qquad \theta\in [0,2\pi].
\]
The function
$\rho_{\theta}(T)$ has the following geometrical interpretation.
Given a bounded convex subset $A$ of $\C$, its \textit{support function}
is defined as
$s_A(\theta)=\sup_{z\in A}\Re (e^{-i\theta} z)$
(so that $A$ is contained in the half-plane
$\big\{\Re (e^{-i\theta} z) \le s_A(\theta)\big\}$, but is not contained in
half-planes $\big\{\Re (e^{-i\theta} z) \le \si\big\}$ for $\si < s_A(\theta)$).
It is easy to see that
\[
\rho_{\theta}(T)=
s_{W(T)}(\theta),
\]
where $W(T)=\{\langle Th, h\rangle: \|h\| = 1\}$ is the numerical
range of $T$ (it is always convex, by the Toeplitz-Hausdorff
Theorem). Notice that $\rho_{\theta}(T)$ is always a continuous
function of~$\theta$.

By \cite[Theorem 17.4]{Trefeth-book},
\begin{equation}
\label{num_range}
\rho_\theta(T)=\lim_{r\to+\infty} r-\Psi_T(re^{i\theta}), \quad \theta\in [0,2\pi].
\end{equation}
The following proposition is a slightly more precise version of
this equality. It will be used in Section~\ref{gen-conv-thms}
below.

\begin{prop}\label{rhot}(cf. \cite[Theorem $17.4$]{Trefeth-book})
Let $ T \in \mathcal B(\mathcal H)$.
Then for any $z= r e^{i\theta}$ with $|z| = r > \rho_{\theta}(T)$,
 \begin{equation}\label{equ}
 |z|-\rho_{\theta}(T)\leq \Psi_{T}(z) \leq
 \sqrt{|z|^2-2\rho_{\theta}(T)|z|+\|T\|^2}\, .
 \end{equation}
\end{prop}

Notice that the inequality $\sqrt{a^2+b} \leq a +\frac{b}{2a}$ (valid for $a>0, a^2+b>0$) gives
\beqn
\label{num_range_mod}
\sqrt{|z|^2-2\rho_{\theta}(T)|z|+\|T\|^2}
\leq
\big(|z|-\rho_{\theta}(T)\big)
 +
\frac{\|T\|^2-\rho_{\theta}(T)^2}{2(|z|-\rho_{\theta}(T))}\, ,
\eeqn
so that the difference between the upper and the lower estimates in
\eqref{equ} tends to $0$ as $|z|\to \infty$.

\begin{proof}[Proof of Proposition \ref{rhot}]
Let $z=r e^{i\theta}$,
$|z|>\rho_\theta(T)$. Then $z\notin\si(T)$, and
\[
(\Psi_{T}(z))^{2}=\inf_{\|h\|=1}\{r^2-2 r \Re\,\langle
e^{-i\theta} Th,  h\rangle+\|Th\|^2\}.
\]
Since $\|Th\|\leq \|T\|$ for all $h$ with $\|h\|=1$, we get
\[
(\Psi_{T}(z))^{2}\leq r^2-2 \rho_{\theta}(T) r +\|T\|^2.
\]

On the other hand, since $\|Th\|\ge \Re\, \langle
 e^{-i\theta} Th,
h\rangle$, we see that
\[
\Psi_{T}(z)^{2}\geq
\inf_{\|h\|=1}(r - \Re\,\langle   e^{-i\theta} Th,  h\rangle)^2 =
(r-\rho_{\theta}(T))^{2},
\]
which gives the first inequality in~\eqref{equ}.
This completes the proof.
\end{proof}

The following lemma estimates the ratio between
the values of $\Psi_{T}$ in two points of the plane.
\begin{lem}\label{epsilon}
Suppose $ T \in \mathcal B(\mathcal H)$. Then
$\frac{\Psi_{T}(z_0)}{|z_0|} \leq \frac{\Psi_{T}(z)}{|z|}(1+\epsilon_{z,z_0})$,
where $\epsilon_{z,z_0}=\frac{\|T\||z-z_0|}{|z_0|\Psi_{T}(z)}$.
\end{lem}
\begin{proof}
For $T \in \mathcal B(\mathcal H)$, we have
\begin{equation}\label{z_0}
(T-z)^{-1}=\frac{z_0 }{z}\,(T-z_0)^{-1}S_{z,z_0},\end{equation}
where $S_{z,z_0}=\frac{z }{z_0}(T-z_0 )(T-z)^{-1}$.
Also,
\beqn
\label{s-z}
\|S_{z,z_0}-I\|
  \le \|(T-z)^{-1}\| \, \big\|\frac{z }{z_0}(T-z_0 )-(T-z)\big\|
   =\frac{\|T\||z-z_0|}{|z_0|\Psi_{T}(z)}
   = \eps_{z,z_0}\, .
\eeqn
Putting together \eqref{z_0} and \eqref{s-z}, we get
\[
\Psi_T(z)^{-1}= \|(T-z)^{-1}\| \leq\|(T-z_0)^{-1}\|\|S_{z,z_0}\|\frac{|z_0|}{|z|}
\leq
\frac{(1+\epsilon_{z,z_0})|z_0|}{\Psi_{T}(z_0)\,|z|}\, .
\]
This completes the proof.
\end{proof}
Using the above Lemma, we will prove the following theorem.
\begin{thm}\label{z by psi}
For any $c>\|T\|$, the restriction $\frac{\Psi_{T}(z)}{|z|}_{\big|
\{|z|\ge c\}}$ is a $\Lip_{\eta(c)}$ function, where $ \eta(c)=
\frac{\|T\| }{c^2} . $ \mg{I changed the formulation}
\end{thm}
\begin{proof}
Take any $z, z_0\in \C$ such that $|z|, |z_0|\ge c$.
By applying twice Lemma \ref{epsilon}, we get
\begin{equation} \label{psi T2}
\frac{\Psi_{T}(z_0)}{|z_0|}-\frac{\Psi_{T}(z)}{|z|}\leq \epsilon\frac{\Psi_{T}(z)}{|z|}
~\rm{~and~}~
\frac{\Psi_{T}(z)}{|z|}-\frac{\Psi_{T}(z_0)}{|z_0|}\leq \delta\frac{\Psi_{T}(z_0)}{|z_0|},
\end{equation}
where
\[
\epsilon=\frac{\|T\||z-z_0|}{|z_0|\Psi_{T}(z)}, \quad
\delta = \frac{\|T\||z-z_0|}{|z|\Psi_{T}(z_0)}.
\]
Therefore,
\begin{align}
\label{psi T3}
\Big|\frac{\Psi_{T}(z)}{|z|}-\frac{\Psi_{T}(z_0)}{|z_0|}\Big|\nonumber
&\leq
\max\Big(\frac{\Psi_{T}(z)}{|z|}\,\epsilon,\frac{\Psi_{T}(z_0)}{|z_0|}\,\delta\Big)
\\\nonumber &
=  \max\Big(\frac{\Psi_{T}(z)}{|z|} \, \frac{\|T\||z-z_0|}{|z_0|\Psi_{T}(z)},
\frac{\Psi_{T}(z_0)}{|z_0|} \, \frac{\|T\||z-z_0|}{|z|\Psi_{T}(z_0)}\Big)
\\\nonumber&= \|T\|
\max\Big\{
\frac{1}{|z_0||z|}, \frac{1}{|z_0||z|} \Big\} \, |z-z_0|
\\\nonumber
&\leq
\frac{\|T\|}{c^2}|z-z_0|
\\&= \eta(c)|z-z_0|
\end{align}
whenever $|z|, |z_0|\ge c$, and we are done.
\end{proof}

We recall the definition of semiconvex functions, see
the book of P. Cannarsa and C. Sinestrari \cite{semiconcave}.
\begin{defn}
Let $A\subset \mathbb R^n$ be an open set and let
$u:A\rightarrow \mathbb R$ be a continuous function.

\begin{enumerate}

\item
We will say that
$u$ is \textit{semiconvex with a constant $C \geq 0$}
if
\[
2u(\mu)-u(\mu+\eta)-u(\mu-\eta)\leq C |\eta|^2
\]
for all  $\mu, \eta\in \R^n$ such that $[\mu-\eta, \mu+\eta]\subset A$.

\item Let $C:A\to \R$ be a positive continuous function. We will
say that $u$ is \textit{semiconvex with bound function $C(x)$} if
for any compact convex subset $B$ of $A$, the restriction $u_{\mid
B}$ is semiconvex with constant $C'=\max_{x\in B} C(x)$.
\end{enumerate}
\end{defn}

\begin{thm}\label{three point}
The function
$\Psi_{T}^{-1}$ is semiconvex on $\C\sm\si(T)$ with bound function
\[
C(z)=2 \Psi_{T}(z)^{-3}.
\]
\end{thm}

\begin{proof}
Let $B$ be a compact convex subset of $\C\setminus \si(T)$, and put
$K=\max_{z\in B} \Psi_{T}^{-1}(z)$.
Suppose that an interval $[\mu - \eta, \mu + \eta]$ is contained
in $B$. Then
\[
2(T-\mu )^{-1}-(T-\mu+\eta)^{-1}-(T-\mu-\eta)^{-1}= -2\eta^2
\,(T-\mu)^{-1}(T-\mu+\eta)^{-1} (T-\mu-\eta)^{-1}\, ,
\]
which implies that
\[
2\|(T-\mu)^{-1}\| \le \|(T-\mu+\eta)^{-1}\|  +
 \|(T-\mu-\eta)^{-1}\| + 2 |\eta|^2 K^3.
\]
This gives our statement.
\end{proof}

Semiconvex functions
admit some interesting characterizations and have good
regularity properties.
We can cite the following facts.

\begin{prop}(see \cite{semiconcave}.)\label{semicon}
Given a continuous function $u :B \rightarrow \mathbb R$ with $B
\subset \mathbb R^n$ open and convex the following conditions are
equivalent:
\begin{enumerate}
\item [(a)] $u$ is semiconvex in $B$ with a semiconvexity constant
$C\geq 0$.

\item [(b)] $u$ satisfies
\[
\label{semiconv}
u(tx+(1-t)y)-tu(x)-(1-t)u(y)
\leq
C\,\frac{t(1-t)}{2}|x-y|^2
\]
for all $x,y$ such that $[x,y] \subset B$ and for all $t \in
[0,1].$

\item [(c)] The function $x\mapsto u(x) + \frac C 2 |x|^2$ is
convex in $B$.
\end{enumerate}
\end{prop}

In particular,
by applying the equivalence of (a) and (c), we get that for any $\la\notin\si(T)$ and
any direction $\zeta\in \C$, $|\zeta|=1$, $\Psi_T^{-1}$ possesses the
one-sided directional derivative at $\la$
\[
\label{der1}
\lim_{s\to 0^+} \frac {\Psi_T^{-1}(\la+s\zeta)-\Psi_T^{-1}(\la)}  s \, .
\]
Hence, the same also holds for $\Psi_T$. It also follows that Alexandroff's theorem applies
to functions $\Psi_T^{-1}$ and $\Psi_T$, so that they are twice differentiable
almost everywhere on $\C\sm\si(T)$.
We refer to \cite[Theorem 2.3.1]{semiconcave} for a precise statement.

It might also be worth recalling here that the function $-\log \Psi_T$ is subharmonic on $\C\sm\sigma(T)$.
Some of the above-stated properties that we state here
are true for Banach space operators. However, there is a difference
between the Hilbert space case
and the Banach space case. For instance, the function $\Psi_T$ can be constant on an open set outside
the spectrum for a Banach space operator, but this cannot happen in the Hilbert space case, see
\cite{{sharg2008}, {DavSharg}} and
references therein.

\section{General theorems on convergence}
\label{gen-conv-thms}
Let  $T$ be a bounded operator
on a Hilbert space.
 The finite section method consists in approximating the spectrum
 of $T$ on a Hilbert space $H$ by spectra of the finite matrices
$T_n=P_nTP_n$, where $\{P_n\}$ is a filtration on $H$.
The possibility of doing it has been studied in several articles.
In
\cite{pokrz}, it is shown that in general, there is no convergence of spectra and
it is determined, for which subsets $K$ of $\BC$ there exists a
filtration $\{P_n\}$ such that $d_H(\sigma(T_n), K)\to 0$ as $n\to \infty$, where
$d_H$ denotes the Hausdorff distance.
On the other hand, there are also some positive results assuring the convergence of
$\sigma(T_n)$ to $\sigma(T)$ under some restrictive hypotheses, see
\cite{Elsner, BandtGuv} and references in
\cite{BandtGuv}. Proposition 4.2 in \cite{boc} contains
an abstract result on the partial limit set
of $\eps$-pseudospectra of $T_n$, under certain hypotheses.
The approach related with $C^*$ algebras, originated in the works by Arveson
\cite{AV,ARV} turned out to be very useful, see
the book \cite{HagRochSilb}. In \cite{brown}, this approach was applied to obtain
positive results for the case of quasidiagonal operators.
We also refer to Hansen \cite{han, hansen2011}, B\"ogli \cite{bog2016} and B\"ogli and
Siegl \cite{bog} and references therein for more results on convergence of spectra for bounded
and unbounded operators. In general, the convergence is only assured if either there
is a kind of norm convergence of $T_n$ to $T$ or if $T$ belongs to a subclass of linear
operators and the filtration $\{P_n\}$ is chosen in a special way.

In this section, we will prove that for a quasitriangular operator $T$
and the corresponding filtration $\{P_n\}$,
the injectivity  radius  $j_{T_n}(z)$ converge uniformly to the injectivity radius
$j_{T}(z)$ on $\mathbb C$.
One of the main results
of this section is Theorem \ref{main thm}, which asserts that
for any
$T\in \mathcal B(\mathcal H)$, there
exists a sequence of matrices $\{S_n\}$ such that the functions $\Psi_{S_n}$ converge
uniformly to  $\Psi_T$ on $\C$. This will be done with the use of the following powerful result.

\begin{thmAFV}[the Apostol--Foia\c{s}--Voiculescu theorem, see \cite{{Apost}, {Aposto}, {Apostol}}]
A Hilbert space operator $T$ is quasitriangular if and only if
$\ind(T-\lambda)\ge 0$ whenever $\la\in\C$ and $T-\lambda$ is Semi-Fredholm.
\end{thmAFV}

We recall that an operator $T \in \mathcal B(\mathcal H)$
is said to be Semi-Fredholm if $\Ran T$ is closed and at least one
of $\ker T$ and $\ker T^*$ is finite dimensional. The index of a
Semi-Fredholm operator $T \in \mathcal B(\mathcal H)$ is defined
by $\ind(T)=\dim \ker T- \dim \ker T^*$.

The following lemma is an inequality between
the injectivity radius of a quasitriangular operator $T$ and the injectivity radius of $T^*$.
\begin{lem}\label{T and T^*}
Suppose $T \in \mathcal B (\mathcal H)$ is quasitriangular. Then
$j_{T}(\la)\leq j_{T^*}(\bar{\la})$ for any $\la \in \mathbb C$.
\end{lem}
\begin{proof}
Suppose that, to the contrary, $j_{T}(\lambda)> j_{T^*}(\bar{\lambda})$ for some $\lambda \in \mathbb C$.
Then there exists an $\epsilon >0,$ e.g. $ \epsilon=\frac{j_{T}(\lambda)}{2}$, such that $\|(T-\lambda)x\|\geq \epsilon\|x\|$ for all $x\in \mathcal H$. Then $(T-\lambda)$
is one-to-one, $(T-\lambda) \mathcal H$ is
closed and $(T^*-\bar{\lambda})\mathcal H= \mathcal H$. Hence $(T-\lambda)$ is semi-Fredholm. Using Theorem AFV, we conclude that $\ind(T-\lambda)\ge 0,$ which implies
that \mbox{$0=\dim \ker(T-\lambda)\geq \dim \ker(T^*-\bar{\lambda}).$} Hence, $(T^*-\bar{\lambda})$ is one-to-one,
which implies that $(T-\lambda)\mathcal H=\mathcal H$. Hence $(T-\lambda)$ is invertible.
Then by Proposition~\ref{Muller}, $j_{T}(\lambda)= j_{T^*}(\bar{\lambda})$, a contradiction.
\end{proof}

\begin{lem}\label{quasi}
Suppose $T \in \mathcal B (\mathcal H)$ is quasitriangular
and $\{P_n\}$ is a corresponding filtration.
Then $j_{T_n}(\lambda)$
converges pointwise to $j_{T}(\lambda)$, where $T_n=P_n T\big|_{P_n\mathcal H}$.
\end{lem}
\begin{proof}
We have $\|(I-P_n)(T-\lambda)P_n\|\rightarrow 0, $ for any $\lambda \in \mathbb C$. Take any $x\in P_n\mathcal H$
such that $\|x\|=1$.
Then
\beqn
\label{Tn}
\begin{aligned}
\|(T_n-\lambda) x\|=\|P_n(T-\lambda) x\|& \geq
\|(T-\lambda) x\|-\|(I-P_n)(T-\lambda)P_n x\|\\
&\geq j_{T}(\lambda)
-\|(I-P_n)(T-\lambda)P_n\|.
\end{aligned}
\eeqn
We get that $\liminf j_{T_n}(\lambda)\geq j_{T}(\lambda)$.
If we can show that $\limsup j_{T_n}(\lambda)\leq j_{T}(\lambda)$, then we are done.

Take any $\epsilon>0$. Then from definition of $ j_{T}(\lambda)$, we have
$\|(T-\lambda)x\|\leq j_{T}(\lambda)+\eps$, for
some $x\in \mathcal H$ with $\|x\|=1$. Since $P_n \rightarrow I$ strongly, given any $\epsilon>0$, there is a positive integer
$N$ such that  $\|x-P_nx\|< \epsilon$ for all $n\geq N$.
Now, for any $n\ge N$, we have
\begin{align*}
\|(T_n - \lambda )P_nx\|\nonumber &=\|P_n(T-\lambda )P_nx\|\\& \leq \|P_n(T-\lambda )x\|+\|P_n(T-\lambda )(x-P_nx)\|\nonumber
\\ & \leq j_{T}(\lambda)+\eps+ \eps\|T-\lambda \|.
\end{align*}
Since $\|P_nx\|\geq 1-\eps$, we get
$j_{T_n}(\lambda)\leq \frac{j_{T}(\lambda)+\epsilon(1+\| T-\lambda \|)}{1-\epsilon}$
for all $n\ge N$. Hence we have $\limsup j_{T_n}(\lambda)\leq j_{T}(\lambda)$. This completes the proof.
\end{proof}

We will need a known analysis fact, which says that the pointwise convergence of functions
implies the uniform convergence under some extra conditions.
\begin{prop}[see \cite{rudin}, Theorem 7.13]
\label{rudin}
Suppose
$K$ is compact and $\{f_n\}$ is an increasing sequence of continuous
functions on $K$
(so that $f_n\le f_{n+1}$ for all $n$). If
$f_n$ converge pointwise to a continuous function $f$ on $K$,
then this convergence is uniform.
\end{prop}
Using this proposition, we will prove the following lemma (which extends  \cite[Theorem 3.9]{brown}).
\begin{lem}\label{rhothta}
Let $T \in \mathcal B (\mathcal H)$ and put
$T_n=P_n T \big|_{P_n {\mathcal H}}$, where
$\{P_n\}$ is an arbitrary filtration on $\mathcal H$.
 Then $\rho_{\theta}(T_n)$ converges to $\rho_{\theta}(T)$ uniformly in $\theta\in [0,2\pi]$. \end{lem}
\begin{proof}
It is easy to see that
$\rho_{\theta}(T_n)\leq \rho_{\theta}(T)$ for all $n$ and that
the sequence
$\{\rho_{\theta}(T_n)\}$ is increasing.
Therefore for any $\theta$, there exists a finite limit $\lim_n \rho_{\theta}(T_n)\le \rho_{\theta}(T)$.
Now, fix some $\theta\in [0,2\pi]$ and
some $\eps>0$. Find $h\in \mathcal{H}$, $\|h\|=1$ such that
$\Re\, \langle e^{-i\theta} T h, h \rangle > \rho_\theta(T)-\eps$.
Put $h_n=P_n h$. Since
\[
\Re\, \langle e^{-i\theta} T h, h \rangle
=\lim \Re\, \langle e^{-i\theta} T h_n, h_n \rangle
=\lim \Re\, \langle e^{-i\theta} T_n h_n, h_n \rangle
\]
and $\|h_n\|\to 1$, we get $\lim_n \rho_{\theta}(T_n)\ge \rho_{\theta}(T)-\eps$.
Hence $\rho_{\theta}(T_n)$ converge pointwise to $\rho_{\theta}(T)$ on $[0,2\pi]$.
By Proposition \ref{rudin}, we conclude that $\rho_{\theta}(T_n)$ converge
uniformly to $\rho_{\theta}(T)$ on $[0,2\pi]$.
\end{proof}
As a consequence of the above lemmas, we will prove the following theorem.

\begin{thm}\label{uniformly pseudo}
Suppose $T \in \mathcal B (\mathcal H)$ is quasitriangular
and $\{P_n\}$ is an associated filtration on $\mathcal H$.
Then $\{j_{T_n}\}$ converges uniformly to $j_{T}$ on $\mathbb C$.
\end{thm}

\begin{proof}
Fix any $R>0$; first we check the uniform convergence on
the closed ball $B_{R}(0)=\{z:|z|\leq R\}$.
To this end, take some $\eps>0$. By compactness, $B_{R}(0)$ has
a finite $\epsilon$-net $\{\lambda_k: 1\le k\le m\}$, so that
$B_{R}(z)\subset \bigcup_{k=1}^{m} B_\eps (\lambda_k)$.

By Lemma \ref{quasi}, there exists
an integer $N$ such that $|j_{T_n}(\lambda_k)-j_{T}(\lambda_k)|<\epsilon$ for all $k$
and all $n \geq N$. Since $\{j_{T_n}\}$ and $j_{T}$ are $\Lip_1(\C)$ functions,
we can now apply a standard $3\eps$ argument.
Namely, let $\lambda$ be any point in $B_{R}(0)$. Then $\la\in B_\eps (\lambda_k)$ for some $k$, and
we get
\[
|j_{T_n}(\lambda)-j_{T}(\lambda)|\leq
\big|
j_{T_n}(\lambda)-j_{T_n}(\lambda_k)\big|
+
\big|j_{T_n}(\lambda_k)-j_{T}(\lambda_k)\big|
+
\big|j_{T}(\lambda_k)-j_{T}(\lambda)\big|
<3\epsilon
\]
for all $n \geq N$. This implies the uniform convergence
on $B_{R}(0)$.

Now we prove the uniform convergence on the whole complex plane.
Once again, fix some $\eps>0$.  Put $R=R(\eps)=\|T\|^2/(2\eps+\|T\|)$.
Since $|\rho_\theta(T)|\le \|T\|$, it follows from~\eqref{equ}, \eqref{num_range_mod}
that for any $z=re^{i\theta}$
with $|z| > R$, one has
\[
\big|j_T(z)-|z|+\rho_\theta(T)\big| \le \frac {\|T\|^2}{2(R-\|T\|)} = \eps.
\]
We get in the same way that
$\big|j_{T_n}(z)-|z|+\rho_\theta(T_n)\big|< \eps$
for all $n$ and all $z, |z|> R$.
Also, by Lemma~\ref{rhothta},
$\rho_{\theta}(T_n)$ converges uniformly to $\rho_{\theta}(T)$ on $[0,2\pi]$ as $n\to \infty$, that is,
there exists a positive integer $N_0$ such that $|\rho_{\theta}(T_n)-\rho_{\theta}(T)|<\epsilon$ for all $n\geq N_0$
and all $\theta$.
This implies that for all $z=re^{i\theta}$ with $|z|>R$ and all $n\geq N_0$, we have
\[
|j_{T_n}(z)-j_{T}(z)|  \leq
\big|j_{T}(z)-|z|+\rho_{\theta}(T)\big|+
\big|j_{T_n}(z)-|z|+\rho_{\theta}(T_n)\big|+
|\rho_{\theta}(T_n)-\rho_{\theta}(T)| <3\epsilon.
\]
Now choose $N_1$ so that $|j_{T_n}-j_T|<3\eps$ on $B_{R(\eps)}( 0)$
for all $n\ge N_1$. Then $|j_{T_n}(z)-j_T(z)|<3\eps$ for all $z\in \C$ whenever
$n\ge \max(N_0,N_1)$. This proves that $j_{T_n}$ converges uniformly to $j_{T}$ on $\mathbb C$.
\end{proof}

\begin{cor}
Let $T$ be an operator on $\mathcal{H}$ such that
either $T$ or $T^*$ is quasitriangular.
Let $\{P_n\}$ be a filtration on $\mathcal{H}$ that is associated to $T$ in the first case and
is associated to $T^*$ in the second case.
Then $\{\Psi_{T_n}\}$ converge uniformly to $\Psi_{T}$ on $\C$, where
$T_n=P_n T_{\mid P_n\mathcal{H}}$.
\end{cor}

Indeed, notice first that
$\Psi_{T_n}(z)=j_{T_n}(z)=j_{T_n^*}(\bar z)$ for all $n$ and all $z$.
Next, $\Psi_{T}=j_T$ if $T$ is quasitriangular and $\Psi_{T}(z)=j_{T^*}(\bar z)$ if
$T^*$ is quasitriangular. So both cases follow from Theorem~\ref{uniformly pseudo}.

It follows that, under the above hypotheses on $T$, given any positive numbers
$\eps_1<\eps<\eps_2$, one gets that
\[
\si_{\eps_1}(T_n)\subset
\si_{\eps}(T)\subset
\si_{\eps_2}(T_n)
\]
for all sufficiently large $n$; in this sense, the pseudospectra
$\si_{\eps}(T)$ can be calculated with an arbitrary precision. Of course,
it would be desirable to
have estimates of the rate of the uniform convergence of
$\Psi_{T_n}$ to $\Psi_{T}$ in some concrete terms.

The finite section method and convergence of pseudospectra for band-dominated
operators is considered in \cite[Chap. 6]{RabRochSilb} for the $\ell^2$ case and in
\cite{SeidSilb} for the case of $\ell^p$. We also refer to \cite{BenArtzi} for a discussion of spectral
approximation for finite band selfadjoint operators.
Herrero introduced several extensions of the notion of quasitriangularity (see \cite{her}),
and it would be interesting to know whether there is kind of extension of Theorem~\ref{uniformly pseudo}
for these classes.

The proof of the following proposition is very easy and we leave details to the reader.
\begin{prop}
Let $T_1, T_2 \in \mathcal B (\mathcal H)$. Then $j_{T_1\oplus T_2}(z) = \min\{j_{T_1}(z),j_{T_2}(z)\}$.
\end{prop}

The following
is one of our main results.

\begin{thm}\label{main thm}
For any bounded linear operator $T$ on a Hilbert space $\mathcal H$, there exists
a sequence $\{T_n\} $ of finite matrices
such that $\Psi_{T_n}$ converges uniformly to $\Psi_{T}$ on $\mathbb C$. \end{thm}
\begin{proof}
Suppose $T \in \mathcal B (\mathcal H)$ and let $J$ be the set of all isolated points of $\sigma(T)$.
Set $K=\sigma(T)\setminus J$. Then $K$ is compact. Let $N$ be any normal operator
on $\mathcal H$ with discrete
spectrum, whose eigenvalues are contained in $K$ and are dense there.

Set $S=T\oplus N$. Then it is very easy to see that
$\sigma(T)=\sigma(S)$.
First we will show that $S$ is quasitriangular.

By Theorem AFV, we have to take an arbitrary point
$\la\in \sigma(S)$
and to show that either $S-\la$ is not Semi-Fredholm or $\ind(S-\la)\ge 0$. To do it, consider
two cases.

\medskip

\noindent{\textbf{Case 1:}} Suppose $\lambda \in K$. Then $\Ran(N-\lambda)$ is not closed.
This implies that $\Ran(S-\lambda )$ is not closed. Hence $S-\lambda $ is not semi-Fredholm.

\smallskip

\noindent{\textbf{Case 2:}} Suppose $\lambda \in J$ and $S-\lambda $ is semi-Fredholm.
There are points $\mu\notin \si(T)$ arbitrarily close to $\lambda$. By stability of the
Fredholm index, we get
$\ind(S-\lambda )=\ind(S-\mu )=0$.

\medskip

Hence we conclude that $S$ is quasitriangular.

Let $\{P_n\}$ be the corresponding filtration on $\mathcal{H}\oplus \mathcal{H}$, so that
$\lim_{n \rightarrow \infty}\|(I-P_n)SP_n\|=0$.
Set $S_n=P_n S \big|_{\mathcal{H}_n}$, where $\mathcal{H}_n=P_n \mathcal{H}$. By applying
Theorem \ref{uniformly pseudo}
and Lemma \ref{psi}, we get  that $\Psi_{S_n}$ converges uniformly to $\Psi_{S}$ on $\mathbb C$.
For any $\la \notin \si(T)$,
$\|(N-\lambda)^{-1}\|= 1/\operatorname{dist}(\lambda, K)\le  \|(T-\lambda)^{-1}\|$,
and therefore
\[
\Psi_{S}(\lambda)^{-1}
=\max
\big\{\|(T-\lambda)^{-1}\|,
\|(N-\lambda)^{-1}\|\big\}
=\|(T-\lambda)^{-1}\|
=\Psi_{T}(\lambda)^{-1}.
\]

Hence we conclude that $\Psi_{S}(\lambda)=\Psi_{T}(\lambda)$
for all $\la\in\C$, which completes the proof.
\end{proof}

\

Notice that for a concrete operator $T$, the above construction
requires the knowledge of the spectrum of $T$, which is
computationally difficult and requires, in general, three passages
to limits (see the work \cite{BenArtziHansen_etl} and its full
version in arxiv, where the smallest number of limit procedures necessary to
solve a computational problem is studied in a systematic way).

Even if the operator $N$ in the last proof is known, it does not
seem so easy to construct the corresponding filtration $\{P_n\}$
on $\mathcal{H}\oplus \mathcal{H}$ explicitly. Therefore, to the
opposite to Theorem~\ref{uniformly pseudo}, the above proof of the
last theorem is not constructive, and a more explicit construction
would be desirable.

\begin{rem} It is well known that pseudospectra varies continuously with an operator $T$
in $\mathcal B(\mathcal H)$. Hansen in his fundamental paper \cite{hansen2011}
introduced the notion of  $(N, \eps)$-pseudospectra
defined by means of a modified function $ \Psi_{T,N}$, which has all the nice continuity property that the
function $\Psi_{T}$ has, but also allow one to approximate the spectrum arbitrarily well for large $N$.
Later,  M. Seidel \cite{seid2012} extended the concept of $(N, \eps)$-pseudospectra of Hansen to the case
of bounded linear
operators on Banach spaces and proved several relations to the usual spectrum.
\end{rem}

We refer to
Part~3 of the book~\cite{PourRichardsBook}, to recent works
\cite{CubittPGarciaWolf, DerevEtl2016, MarlNaboko}
and references therein
for diverse negative and positive results on
computability of spectra and on the rate of
convergence of approximations.

\section{A theorem about shapes of pseudospectra}
\label{sect-shapes}

Let $\Om$ be a bounded domain in $\C$. We put
$\ovr \Omega = \{\bar{w}\in \mathbb C: w\in \Omega\}$.
Let $H^2(\ovr\Om)$ stand for
the Hardy space on $\ovr\Om$. We define the subnormal
operator $M(\ovr\Om)$ of multiplication by
$z$ on $H^2(\ovr\Om)$ by
\[
M(\ovr\Om)f(z)=zf(z), \quad f\in H^2(\ovr\Om).
\]

If $ G\supset\clos \Omega$, then we write it as $G\Supset \Omega$.

\begin{thm}\label{nilpotent}
Suppose $G_0\Supset G_1 \Supset \ldots \Supset G_ m$
are bounded connected domains in $\C$, containing the origin,
and $\epsilon_1$ is any number such that
$\eps_1 <\dist(\partial G_0, \partial G_1)$.
Then there
exist $\epsilon_2, \epsilon_3,\ldots, \epsilon_m$ and a
square nilpotent complex matrix $T$ such that $ \epsilon_1 > \epsilon_2 > \ldots > \epsilon_m>0$ and
\beqn
\label{ps-inclusn}
G_0 \Supset\sigma_{\epsilon_1}(T)\Supset G_1\Supset \ldots
\Supset G_{m-1} \Supset \sigma_{\epsilon_m}(T)\Supset G_m.
\eeqn
\end{thm}

\begin{proof}
Take any finitely connected domains $\Om_j$, $j=1, \dots, m$
with smooth boundaries, such that
$G_0\Supset \Om_1 \Supset G_1 \Supset \ldots
\Supset \Om_m\Supset G_m$. We will assume that $\Om_1$ is close to $G_1$, so that
$\eps_1 <\dist(\partial G_0, \partial \Om_1)$.
Put
\[
T_j(N)=M(\ovr\Om_j)^*|\ker M(\ovr\Om_j)^N.
\]
Operators
$T_j(N)$ are nilpotent for any $N$.
We will show that the pseudospectra of the operator
\beqn
\label{def-T}
T=\oplus_{j=1}^m T_j(N_j)
\eeqn
(acting on a finite dimensional Hilbert space)
satisfy the inclusions \eqref{ps-inclusn}
if the numbers $N_j$ and $\eps_j$
are properly chosen.

These numbers will be defined by an inductive construction. We set
\[
\de=\min_{j, k}\,\dist(\pt\Om_j, \pt G_k)>0.
\]

$\bullet$ \enspace First step:
Notice that $M(\ovr\Om_1)$ is subnormal and $\si(M(\ovr\Om_1)^*)=\clos \Om_1$.
Theorem~\ref{uniformly pseudo}
implies that $\Psi_{\,T_1(N)}(z)\to \dist(z, \Om_1)$
on $\C\sm\Om_1$ and
$\Psi_{\,T_1(N)}(z)\to 0$
on $\Om_1$ as $N\to \infty$, uniformly in both cases.
Choose $N_1$ so that
$\Psi_{T_1(N_1)}(z)>\eps_1$ for $z$ in $\C\sm G_0$
and $\Psi_{T_1(N_1)}(z)<\eps_1/2$ on~$\Om_1$.

$\bullet$ \enspace $k$th step ($2\le k\le m$):
Suppose $N_1, \ldots, N_{k-1}$  and $\eps_j$ $(2\le j\le k-1)$ have been elected already.
On this step, we choose $\eps_k$ and $N_k$.

Choose any $\eps_k$ so that $\eps_k<\eps_{k-1}$, $\eps_k<\de$, and
\[
\max
\big\{
\|(T_j(N_j)-z)^{-1}\|: \quad
1\le j\le k-1, \;
z\in \C\sm G_{k-1}
\big\} < \eps_k^{-1}.
\]
Notice that
$\Psi_{\,T_k(N)}\to 0$ uniformly
on $\Om_k$ as $N\to \infty$. Choose $N_k$ so that
$\Psi_{\,T_k(N_k)}\le \eps_k/2$ on $\Om_k$.

After
$\eps_2, \dots, \eps_m$ and
$N_1, \dots, N_m$
have been chosen, define $T$ by \eqref{def-T}. It is a nilpotent operator on a finite dimensional
Hilbert space.

If $1\le k\le j\le m$
and $z\in\C\sm G_{k-1}$, then
\[
\|\big(T_j(N_j)-z\big)^{-1}\|
\le
\|(M(\ovr\Om_j)^*-z)^{-1}\|
\le \frac 1 {\dist(\pt G_{k-1}, \Om_j) }
\le \frac 1\de < \eps_k^{-1}.
\]
It follows that
$
\max_{z\in \C\sm G_{k-1}}\|(T_j(N_j)-z)^{-1}\|
 < \eps_k^{-1}
$
for all $j=1, \ldots, m$,
so that
\[
\max_{z\in \C\sm G_{k-1}}\|(T-z)^{-1}\|
 < \eps_k^{-1}.
\]
This implies that
$\si_{\eps_k}(T)\subset G_{k-1}$.
On the other hand,
$\Psi_{\,T_k(N_k)}\le \eps_k/2$ on $\Om_k$ implies that
$\Psi_T\le \eps_k/2$ on $\Om_k$, so that
$\si_{\eps_k}(T)\Supset\Om_k\supset G_k$ for $k=1, \ldots, m$
(recall that $\Psi_T$ is continuous on $\C$).
It follows that $T$ satisfies all inclusions in \eqref{ps-inclusn}.
\end{proof}

\begin{rem}\label{arbitr-shape}
The inclusions given in \eqref{ps-inclusn} imply that for any
$\eps\in [\eps_m, \eps_1]$, there exists an index $j$, $0\le j\le m-2$
such that $G_{j+2}\subset \si_{\eps}(T) \subset G_j$.
So Theorem~\ref{nilpotent} shows that in some sense, the shape of pseudospectra of a finite matrix
can be arbitrary. Certainly, we only are able to exhibit the example of this kind by
taking the quotients $\eps_j/\eps_{j+1}$ very large. As we mention in the Introduction,
the problem of describing all possible functions $\Psi_T(z)$ remains open.
\end{rem}

\section{Multipliers}
\label{sect-multipls}

\subsection{Cowen-Douglas class and estimates of functions of nilpotent matrices}
First let us recall the well known class of operator from the fundamental
paper of Cowen-Douglas \cite{cd}.
\begin{defn}
For $\Omega$ a connected open subset of $\mathbb C$ and $m$ a
positive integer, let $\mathbf B_{m}(\Omega)$ denote the set of
operators $T$ in $\mathcal B(\mathcal H)$ which satisfy the
following properties:
\begin{itemize}

\item  $\Omega\subset \sigma(T)$,
\item $\Ran {(T-w)}=\mathcal H $  for all $w$ in $\Omega$,

\item $\bigvee_{w\in\Omega}\ker(T-w)= \mathcal H$,
\item $\dim~\ker(T-w)= m \mbox{ for } w\in\Omega$. \end{itemize}
\end{defn}

Suppose $T$ is in $\mathcal B_m(\Omega)$ and $0\in \Om$. Put $
\cH_n=\ker T^n, \quad T_n=T\big|_{\mathcal H_n}$ and
$P_n=P_{\mathcal H_n}$. Then $\{P_n\}$ is a filtration on
$\mathcal H$ and
 $T$ is quasitriangular with respect to this filtration.
As we will see a little bit later,  $\cH_n$ is finite dimensional
and $T_n$ is nilpotent for any $n$.

\medskip

Let $f$ be a function, defined and analytic in some (connected) neighborhood of $0$.
Then all operators $f(T_n)$ are well defined.
Notice also that $\cup \, \cH_n$ is dense in $\cH$
(we refer to \cite[Section 1]{cd} for a background).
Put $f^*(z)=\ovr{f(\bar z)}$. The main result
of this section, Theorem~\ref{thm-mult}, says that
the norms $\|f(T_n)\|$ are uniformly bounded if and only if $f^*$
is a germ of a function in a certain multiplier space.
This will motivate Example~\ref{ex-mult}.

First we will need some preliminaries.

Put $\mathcal N_{\lambda}= \ker (T -\lambda )$, where $\la\in \Om$.
As it follows from the Grauert theorem,
the family of spaces $\{\mathcal N_\lambda\}_{\la\in \Om}$ possesses a
\textit{global analytic frame}: there exist analytic functions
$\ga_j:\Om\to\cH$, $1\le j\le m$ such that
$\{\ga_j(\la):1\le j\le m\}$ is a basis of $\cN_\la$ for any $\la\in \Om$
(see \cite{cd}).
Let
$\rho(\la):\BC^m\rightarrow \cN_\la$ be the isomorphism, defined by
\[
\rho(\la) e_j = \ga_j(\la), \qquad j=1, \ldots, m
\]
(here $\{e_j\}$ is the standard basis of $\BC^m$).
Then $T\rho(\la)=\la\rho(\la)$, $\la\in\Om$.

The following proposition is rather standard.
For reader's convenience, we include a simple proof.

\begin{prop}
There exists a Hilbert space $\wt\cH$ of holomorphic functions
from $\ovr \Omega = \{\bar{w} : w\in \Omega\}$
to $\BC^m$ and an isometric isomorphism
$
V: \cH\to\wt\cH
$
such that
\[
T^* = V^{-1}M_z V,
\]
where
 $M_z$ is the operator of multiplication by the co-ordinate
function on $\wt\cH$.
\end{prop}

\begin{proof}
This realization is provided by the injective map $V: \cH\to\Hol(\ovr\Om, \BC^m)$, given by
\[
Vx(\la)= \big(\rho(\bar\la)\big)^* x, \qquad x\in \cH,\, \la \in \ovr\Om
\]
(here $\Hol(\ovr\Om, \BC^m)$ stands for the space of all analytic function from
$\ovr\Om$ to $\BC^m$).
The identity
$T\rho(\la)=\la\rho(\la)$ implies
the intertwining property $VT^* = M_z V$, so that
one has just to set $\wt\cH= V\cH$.
\end{proof}

Certainly, $\wt\cH$ can be seen as a vector-valued
reproducing kernel Hilbert space.
This is, in fact, an alternative point of view
to the Cowen--Douglas class, which is discussed in the paper of Curto and Salinas
\cite{curto}; in fact, $k$-tuples of operators were considered there.
For one operator $T\in \mathcal{B}_1(\Om)$, this fact is contained in \cite[Subsection 1.15]{cd}.

The commutant of $T$ is the weakly closed algebra of operators which commute with $T$.
We denote it as $\{T\}'$.
Notice that for any $S\in\{T\}'$, $S\cN_\la\subset \cN_\la$, $\la\in\Om$.
So there exist (uniquely defined) linear fibre maps
$\Phi_S(\lambda):\mathcal N_{\lambda}\rightarrow \mathcal N_{\lambda}$
such that
\[
Sk =\Phi_S(\la)k \quad \text{for all $\la\in \Om$, $k\in \cN_\la$}
\]
(Cowen and Douglas in \cite{cd} use the notation $\Phi_S=\Gamma_T S$).

Given an operator $S\in \{T\}'$, put
\[
\wt\Phi_S(\la) = \rho(\la) \Phi_S(\la)\rho(\la)^{-1},
\]
so that $\wt\Phi_S(\la)\in \mathcal{B}(\BC^m)$ and
$\wt\Phi_S(\la)$ is analytic in $\la$.
It is easy to see that
the matrix-valued function $\wt\Phi_S$ is analytic in $\Om$.

\textit{The multiplier algebra} $\Mult\subset \Hol(\ovr\Om)$ is defined as
the set of (scalar) functions $\phi$ on $\ovr\Om$ that multiply $\wt\cH$ into itself,
i.e. \[
\{\phi: \phi f\in \wt\cH,
\mbox{ for all  $f\in \wt\cH$}\}.
\]
It follows from the closed graph theorem that if $\phi$ is a multiplier,
then $M_\phi$ is a bounded linear operator on $\wt\cH$.

The following fact follows immediately.

\begin{prop}
\label{prop-comm}
Consider a subclass of the commutant, defined by
\[
\mathcal C_{T}(\Om)=\{S\in \{T\}':
\mbox{ $\exists$ a scalar function $\phi_{S}(\la)$: }
\Phi_S(\la)=\phi_{S}(\la)I_{\mathcal N_{\la}},\; \la\in\Om\}.
\]
Then
\[
\big\{\phi_S^* : S\in \mathcal C_{T}(\Om)\big\} = \Mult.
\]
\end{prop}

\begin{lemA}(see \cite{cd}, Lemma 1.22). Suppose $ T \in \mathbf B_m(\Omega)$. Then
\begin{enumerate}

\item
\label{labi}
$(T-\la)\ga^{(\ell)}_k(\la)= \ell \ga^{(\ell-1)}_k(\la)$ for all $\ell\ge 1$ and
$k=1, \dots, m$;

\item For all
$\la \in \Om$,
\begin{align*}
\ker(T-\lambda)^n & =
\Span\{\ga_k^{(\ell)}(\la):1\le k\le m, \, 0\leq \ell\leq n-1\} \\
& =\Span\{\Ran \rho^{(\ell)}(\la):0\leq \ell\leq n-1\}.
\end{align*}
\end{enumerate}
\end{lemA}

It follows from this lemma that
the vectors
$\ga_k^{(\ell)}(\la)$ $(1\le k\le m, \, 0\leq \ell\leq n-1)$
form a basis of $\ker(T-\lambda)^n$ for any $\la\in\Om$ and any $n$.
In particular, $T_n$ is nilpotent, and its Jordan form has $m$ Jordan blocks
of order $n$.

\begin{lem}
\label{lem-Phi}
Suppose $S\in \{T\}'$ and $n\ge 1$.
Then $S_{ \mid \cH_n}=0$ if and only if
$\Phi_S=z^n \Psi$,
where $\Psi$ is analytic in a neighborhood of $0$.
\end{lem}

\begin{proof}
By Lemma A, $S_{ \mid \cH_n}=0$ if and only if $S\rho^{(k)}(0)=0$ for $k=0, 1, \ldots, n-1$.

By taking $k$th derivative in the identity $S\rho(\la)=\rho(\la) \wt\Phi_S(\la)$, we get
\[
S\rho^{(k)}(\la)=
\sum_{\ell=0}^{k-1} {k\choose \ell}\rho^{(k-\ell)}(\la)\wt\Phi_S^{(\ell)}(\la) + \rho(\la)\wt\Phi_S^{(k)}(\la).
\]
By applying induction in $k$, we get
that $S_{ \mid \cH_n}=0$ if and only if
$\wt\Phi_S^{(k)}(0)=0$ for all $k=0, 1, \ldots, n-1$
(notice that $\rho(0)$ is an isomorphism).
This implies the statement of Lemma.
\end{proof}

\begin{prop}
\label{prop-restr}
If $S\in \mathcal C_{T}(\Om)$, then
$S_{\mid \cH_n}= \phi_S(T_n)$ for all $n\ge 1$.
\end{prop}

\begin{proof}
Fix some $n\ge 1$, and let $g$ be any polynomial such
$\phi_S-g = z^n\psi$, where $\psi$ is analytic at $0$. Then $\Phi_{S-g(T)}=z^n\psi(z) I_{\mathcal{N}_z}$.
By Lemma~\ref{lem-Phi},
$(S-g(T))_{\mid H_n}=0$, and therefore
$S_{\mid H_n}=g(T)_{\mid H_n}=\phi_S(T_n)$.
\end{proof}

As consequence of the above lemma, we will prove the following theorem.
\begin{thm}
\label{thm-mult}
Let  $0\in \Om$ and let  $ T$ be an operator in $\mathbf
B_n(\Omega)$. Put $\mathcal H_n=\ker T^n$, $T_n=T_{ \mid \mathcal H_n}$.
Let $f$ be a function, defined and analytic in the neighborhood of $0$.
Then the following properties are equivalent.
\begin{enumerate}
\item The norms $\|f(T_n)\|$ are uniformly bounded;

\item There exists an operator $S\in \mathcal C_T(\Om)$
such that $\phi_S=f$;

\item $f^*\in \Mult$ (or, more precisely, $f^*$ extends to a function in $\Mult$).
\end{enumerate}

If these properties hold, then $S^*$ equals to the multiplication by $f^*$ on $\wt\cH$.
\end{thm}

\begin{proof}
Since $T_n$ acts on a finite-dimensional $\cH_n$ and is nilpotent,
$f(T_n)$ is well-defined for all $n$. It is immediate that
$f(T_n)_{\mid H_m}=f(T_m)$ for all $n\ge m$, therefore the norms
$\|f(T_n)\|$ increase as $n\to\infty$. First we show that (1) and (2)
are equivalent.

(2)$\implies$(1). \enspace
Let $S\in \mathcal C_{T}(\Om)$. Then, by Propositions~\ref{prop-comm} and
~\ref{prop-restr}, $\|\phi_S(T_n)\|=\|S_{\mid \cH_n}\|\le \|S\|$ for all $n$.

(1)$\implies$(2). \enspace
Suppose that the norms
$\|f(T_n)\|$ are uniformly bounded. Since $\cup \,\cH_n$
is dense in
$\cH$, the formula
\[
S_{\mid \cH_n}=f(T_n), \quad n\ge1
\]
defines correctly a bounded operator $S$ on $\cH$.
For
any $h \in \mathcal \cH_n$,
we have
\[
STh=
f(T_n)T_n h =T_n f(T_n)h =T_n Sh =TSh.
\]
Hence, $S\in \{T\}^{\prime}$. Now we can repeat the arguments used above in the proof of Proposition~\ref{prop-restr}.
Fix some $n\ge 1$, and let $p_n$ be a polynomial  such that $f-p_n=z^n \psi$, where $\psi$
is analytic at $0$.
We have
\[
S_{\mid \cH_n}=f(T_n)=p_n(T_n)=p_n(T)_{\mid \cH_n}
\]
(the equality $f(T_n)=p_n(T_n)$ is due to the Jordan structure of $T_n$).
By Lemma~\ref{lem-Phi}, this implies
that $\Phi_{S}-\Phi_{p_n(T)}=z^n\Psi$ for an analytic fibre map $\Psi$.
Since $\Phi_{p_n(T)}(z)=p_n(z) I_{\cN_z}$, we get that for any $n$,
$\Phi_{S}(z)$ coincides with $f(z) I_{\cN_z}$ at the origin, up to the $n$th order. Therefore
$\Phi_{S}(z)=f(z) I_{\cN_z}$ in a neighbourhood of $0$, which gives (2).

The equivalence (2)$\Longleftrightarrow$(3) and the last statement of the Theorem follow from Proposition~\ref{prop-comm}.
\end{proof}

\subsection{An example}
In what follows, we will denote by $\sqrt{\phantom\cdot}$ \textit{the principal branch} of the square
root, defined for all $z\ne 0$ by $\sqrt z= |z|^{1/2}\exp\big(i (\arg z)/2\big)$, where
$\arg z \in (-\pi, \pi]$ (so that the cut is along $\R_-$).

The next Lemma is auxiliary and will be used in Example~\ref{ex-mult} below.

\begin{lem}\label{taylor}
Define the function
$f_t(z)=\root\of{z^2-z+t}$, where $t\in\BC$, and let
\[
f_t(z)=\sum_{n=0}^\infty \hat f_n(t) z^n
\]
be its Taylor expansion at the origin.
Let $0<r<1/4$ and $\MM>0$ be fixed.
Then there exists some
$N=N(r,M)$ such that
\[
\max_{1\le n\le N} |\hat f_n(t)| >M
\]
for all $t\in\BC$ such that $|t-1/4|=r$.
\end{lem}

\begin{proof}
Fix some radius $r\in (0, 1/4)$, and let $|t-1/4|=r$. The roots of
$z^2-z+t$ are $z_{1,2}=z_{1,2}(t)=1/2\pm \,\sqrt{1/4-t}$. Notice
that $|z_{1,2}|<1$.
 Set
\[
f_t(z) = \sqrt{z_1(t)-z} \;\sqrt{z_2(t)-z}.
\]
This coincides with the previous definition in a neighbourhood of $0$, but now
$f_t(z)$ turns to be holomorphic in the disc $|z|<\rho(t)$, where $\rho(t)=\min(|z_1|, |z_2|)$.
We have
\[
f_t(z)=g_t(z)+h_t(z),
\]
where
\begin{align}
g_t(z)&=\big(\sqrt{z_1-z}-\sqrt{z_1-z_2}\big)\,\big(\sqrt{z_2-z}-\sqrt{z_2-z_1}\big), \\
\label{h} h_t(z)&=\sqrt{z_2-z_1}\; \sqrt{z_1-z} +
\sqrt{z_1-z_2}\; \sqrt{z_2-z} - \sqrt{z_2-z_1}\; \sqrt{z_1-z_2} \, .
\end{align}
We denote by $\hat g_n(t), \hat h_n(t)$ the Taylor coefficients of
$g_t(z)$ and $h_t(z)$.

One has the formula
\begin{equation}
\label{tayl-sqrt} \sqrt{1-z}=1-\sum_{n=1}^\infty c_n z^n, \quad
|z|<1,
\end{equation}
where
\[
c_n=\frac 1 {2(n+\frac 12)(n-\frac 12)B(\frac12, n+1)} \sim \frac
1 {2\sqrt\pi\, n^{3/2}}\, , \quad n\to \infty
\]
%
%
($B$ is the Beta function).
By \eqref{h} and \eqref{tayl-sqrt}, we get
\[
\hat h_n(t) = - c_n( a(t) z_1(t)^{-n} + b(t) z_2(t)^{-n} ), \quad n\ge 1,
\]
with
\[
a(t)= \sqrt{z_2(t)-z_1(t)}\, \sqrt{z_1(t)}, \quad b(t)=
\sqrt{z_1(t)-z_2(t)}\, \sqrt{z_2(t)}\, .
\]
Since $|z_1(t)-z_2(t)|=2\,\sqrt r$, it is easy to see that there
is a constant $\eps>0$, independent of $t$, such that for any
$n\ge 1$,
\[
\qquad \qquad
\big|a(t) z_1(t)^{-n+1} + b(t) z_2(t)^{-n+1} \big| + \big|a(t)
z_1(t)^{-n} + b(t) z_2(t)^{-n} \big|
>\eps \rho(t)^{-n}\, .
\]
Therefore $|\hat h_{n-1}(t)|+|\hat h_n(t)|> \eps'n^{-3/2}\rho(t)^{-n}$,
where $\eps'=\eps'(r)>0$.

We assert that a similar lower estimate holds for the Taylor coefficients of $f_t$.
To show it, consider the two-point set
\[
E(t)=\big\{ \rho(t)\, \frac {z_1(t)}{|z_1(t)|}, \rho(t)\, \frac
{z_2(t)}{|z_2(t)|} \big\}\, .
\]
Then for $|z|<\rho(t)$, $|g_t''(z)|\le K \dist\big(z,
E(t)\big)^{-1/2}$, where $K$ does not depend on $t$ (here
$g_t''(z)=d^2g_t(z)/dz^2$). Hence $\|g_t''\|_{H^1(B_{\rho(t)}(0))}\le
K_1$ ($H^1$ stands for the Hardy space). This gives
\[
|\hat g_n(t)|\le \frac {K_2}{n^2}\, \rho(t)^{-n} .
\]
%
%
The constants $K_1$ and $K_2$ only depend on $r$. Fix any positive constant $\MM$.
Since $\hat f_n(t)=\hat g_n(t)+\hat h_n(t)$, there exists a large $N=N(r,\MM)$ such
that
\[
|\hat f_{N-1}(t)|+|\hat f_N(t)|
>
\bigg( \frac {\eps'}{N^{3/2}}- \frac {2 K_2}{(N-1)^2} \bigg)
\rho(t)^{-N}
>
\frac {\eps'}{2N^{3/2}} > 2 \MM
\]
for all $t$, $|t-1/4|=r$. This implies the statement of Lemma.
\end{proof}

\begin{ex}
\label{ex-mult}
Given any real $r$, $0<r< \frac 12$ and any (large) real number $\MM$, there exists a
nilpotent square matrix $S$, whose size depends on $r$ and $\MM$,
such that
$\|\sqrt{I-S}\|\le 3$, whereas
$\|\sqrt{\tau -S}\|\ge \MM$ for any $\tau$ on the circle $|\tau -1|=r$.
Here $\sqrt{\tau-S}$ is understood in the sense of the Riesz-Dunford calculus,
applied to the function $\sqrt{\tau -z}$, where the principal value of the square root
is meant.
\end{ex}

Indeed, consider the $N\times N$ nilpotent lower triangular Toeplitz matrix
\[
S_N=
{\tiny
\begin{pmatrix}
0          &     0     &     0       &    \cdots \;  &     0 &     0   \\
4          &     0     &     0       &   \cdots \;  &     0 &     0   \\
-4          &     4     &     0       &   \cdots \;  &     0 &     0   \\
0          &     -4     &     4       &    \cdots \;  &     0 &     0   \\
\vdots     &  \vdots   &  \ddots     &   \ddots  & \vdots &    \vdots        \\
0          &     0     & \cdots      &     -4               &     4 &    0
\end{pmatrix}
}\, ,
\]
(which has entries $4$ on the first diagonal under the main one, entries $-4$ on
the second diagonal and all other entries equal to $0$).
We assert that one can put $S=S_N$, where $N=N(r,\MM)$ is sufficiently large.
To see this, notice
first that $S_N=4J_N-4J_N^2$, where $J_N$ is the standard
$N\times N$ Jordan block with ones on the first diagonal under the main one.
It is standard that for any function $\phi$, analytic in a neighbourhood of zero,
$\phi(J_N)$ is well-defined and is a Toeplitz lower triangular matrix, whose
entries in the first column are $\hat \phi_0, \hat \phi_1, \dots, \hat \phi_{N-1}$.
Define $f_t(z)$ as in Lemma~\ref{taylor}.
It follows that
\[
\sqrt{\tau -S_N}=2 f_{\tau/4}(J_N).
\]
In particular, $\sqrt{I-S_N}=I-2J_N$. Therefore
$\|\sqrt{I-S_N}\|\le 3$. Take any $\MM>0$. By Lemma~\ref{taylor},
there is some $N$ such that for any $\tau $ on the circle $|\tau -1|=r$, the matrix
$\sqrt{\tau -S_N}$ has an entry, whose absolute value is greater than $\MM$.
This implies our assertion.

\medskip

Notice that in fact, the above argument proves that for a fixed $r\in (0,1/2)$,
\[
\min_{|\tau-1|=r}\|\sqrt{\tau -S_N}\|
\]
grows exponentially as a function of the size $N$.
The informal explanation of this example is that in the limit (as $N\to \infty$), the matrices
$S_N$ behave as the Toeplitz operator $T_\psi$
with the analytic symbol $\psi(z)=4z-4z^2$ on $H^2\big(B_1(0)\big)$. Then for $\tau=1$,
the square root $\sqrt{I-T_\psi}$ exists as a bounded operator (and
equals to $T_{1-2z}$), whereas a bounded operator square root
$\sqrt{\tau -T_\psi}$ does not exist if $\tau\ne 1$ is close to $1$.
We observe that the spectrum of the ``limit operator'' $T_\psi$ is no longer
one point, instead, it contains a neighbourhood of $1$.

This example also implies that even for a nilpotent matrix $S$,
the values of $\|\sqrt{\tau -S}\|$ can change very rapidly
for $\tau$ in a neighbourhood of $1$. (Notice that
for a fixed $S$, $\sqrt{\tau -S}$ is analytic in this neighbourhood.)
In particular, to the contrary to Corollary~\ref{lips3},
any estimate of the Lipschitz constants of the functions
$\tau\mapsto \|\sqrt{\tau-S}\|$,
$\tau\mapsto \|\sqrt{\tau -S}\|^{-1}$
should depend on the size of $S$.

\subsection{Final remarks on estimates of functions of operators and matrices}
Here we discuss some relations between known results.

Suppose we have an operator $T$ on a Hilbert space $\mathcal{H}$ (which can be
finite dimensional) and suppose that the function $\Psi_T$ is known. One can ask,
what can be said about the
norms $\|f(T)\|$, where $f$ is analytic on $\si(T)$.
This question was raised in the work \cite{GrTref}, which contains an example of two
matrices $T_1$ and $T_2$ with simple eigenvalues and identical pseudospectra
(that is, satisfying $\Psi_{T_1}(z)=\Psi_{T_2}(z)$ for all $z$)
and such that
$\|T_1^2\|\ne \|T_2^2\|$. The matrix norms here and in the definition of
$\Psi_{T_j}$ are induced by the Euclidean norm.
This question was further investigated a series of papers by Ransford
and his coauthors.
The paper \cite{ransf-rost} by Ransford and Rostand gives another example
of such type of matrices with simple eigenvalues.
Moreover,
the two matrices in this latter example have super-identical pseudospectra in the sense that
all singular numbers of $T_1-z$ coincide with those of $T_2-z$, for any
$z\in \BC$.

By a theorem in \cite{ransf-raou}, given a domain $\Om$ and a function
$f\ne \operatorname{const}$ in $\Hol(\Om)$, which
is not a M\"obius transformation, for any $N\ge 6$ and any $M>1$ one can find
$N\times N$ matrices $T_1$ and $T_2$ with identical pseudospectra such that
$\|f(T_1)\|\ge M \|f(T_2)\|$.
On the other hand, it is known (see \cite{forti-ransf}) that,
given matrices $T_1$ and $T_2$ of size $N\times N$ with super-identical pseudospectra,
one has
\[
N^{-1/2}\le \frac {\|f(T_1)\|}{\|f(T_2)\|} \le N^{1/2}
\]
for any function $f$ holomorphic on $\si(T_1)=\si(T_2)$.
It is not known whether there is an estimate independent of $N$.

There are also many other positive results on the estimation of
functions of operators and matrices.
For instance, the following assertion follows from
the main result of \cite{BadeaBeckCrouzeix}.
\begin{thmnonumber}[\cite{BadeaBeckCrouzeix}]
Let $T$ be a Hilbert space operator and let
$z_1, \dots, z_n$ be
points outside its spectrum. Then for any
bounded analytic function on the
(unbounded) domain
$\Om= \BC\sm \cup_j \clos B\big(z_j, \Psi_T(z_j)\big)$,
one has
$
\|f(T)\|\le K\sup_\Om |f|,
$
where $K = n+n(n-1)/\sqrt{3}$.
\end{thmnonumber}

Notice that here $K$ does not depend on the dimension
of $\mathcal{H}$.

Many other results have this form. For instance, suppose
$T$ is a Hilbert space operator, $\si(T)\subset B_1(0)$ and
$\Psi_T(z)\ge r$ for any $z$ on the circle $|z|=1+r$. Then $T$ is
a $\rho$-contraction for $\rho=2+1/r$, which implies
the estimate $\|f(T)\|\le \rho \sup_{B_1(0)} |f|$, for any
function $f$ holomorphic in $B_1(0)$ such that $f(0)=0$ (see
\cite[Section I.11]{SzFBK2010}).
It is easy to describe the numerical range of $T$ in terms of the behavior of
the function $\Psi_T$, see \eqref{num_range}.
Therefore the variant of the von Neumann inequality given by
B. Delyon and F. Delyon in \cite{Delyon_B_F} can also be seen as a positive result
in this direction. We refer to \cite{DriEstYaku} for a generalization
of the result of \cite{Delyon_B_F} to certain non-convex sets associated with the operator.

As positive results on estimation of norms
 $\|f(T)\|$, one can mention the Kreiss matrix
theorem (see, for instance, \cite[Section 18]{Trefeth-book}) and
the results by Szehr and Zarouf (see \cite{Szehr2014,Szehr_Zar} and references therein).

One can also relate the estimates of functions of an operator
with the so-called weak resolvent
sets. By definition (see \cite{FongNordgRadjRos}), an analytic function
on $\BC\sm \si(T)$ is called \textit{a weak resolvent} of a
bounded operator $T$ on a Banach space $\mathcal X$ if it has the
form $z\mapsto G\big((T-z)^{-1}f\big)$ for some $f\in\mathcal{X}$
and $G\in\mathcal{X}^*$. The weak resolvent set $WR(T)$ of $T$ is
the set of all its weak resolvents. This interesting notion was
introduced in 1987 in a paper by Nordgren, Radjavi and Rosenthal
and further studied by Fong and the named three authors in
\cite{FongNordgRadjRos}.
Since it makes no difference, let us consider the Banach space setting.

Let $T_j\in \mathcal{B}(\mathcal{H}_j)$, $j=1,2$ be two Banach space operators.
Following \cite{FongNordgRadjRos}, we say that
$WR(T_1)\subset WR(T_2)$ if $\si(T_1)\subset\si(T_2)$ and each function in
$WR(T_1)$ is also in $WR(T_2)$. Let us cite the following result.

\begin{thmnonumber}[\cite{FongNordgRadjRos}, Theorem 2.8]
If $\si(T_1)$ has finitely many holes and $WR(T_1)\subset
WR(T_2)$, then there is a constant $k$ such that $\|\phi(T_2)\|\le
k \|\phi(T_1)\|$ for any function $\phi$, holomorphic on a
neighbourhood of $\si(T_1)$.
\end{thmnonumber}

In particular, it follows that $\Psi_{T_1}\le k^{-1} \Psi_{T_2}$
on $\C\setminus \si(T_1)$. If $\si(T_1)=\si(T_2)$ and the weak resolvent sets
of $T_1$ and $T_2$ coincide, then one has a two-sided estimate
$\Psi_{T_1} \asymp \Psi_{T_2}$ on $\BC\sm \si(T_1)$.

One can observe that the statement from the above theorem is much
stronger than just the relation $\Psi_{T_1} \asymp \Psi_{T_2}$. In
fact, it is also proven in \cite{FongNordgRadjRos} that
whenever the sets $WR(T_1)$ $WR(T_2)$ coincide in a neighbourhood of
$\infty$, operators $T_1$ and $T_2$ generate isomorphic uniformly
closed algebras. If, moreover, both operators are strictly cyclic,
then they are similar.

\vskip.3cm

\textsl{Acknowledgements:}
The research by A. Pal has been
supported by
a
Marie Curie International Research Staff Exchange
Scheme Fellowship within the 7th European Union Framework
Programme (FP7/2007-2013) under grant agreement n${}^o$
612534,
project
MODULI - Indo European Collaboration on Moduli Spaces.
D. Yakubovich was supported by the project
MTM2015-66157-C2-1-P
of the Ministry of Economy and Competitiveness
of Spain
and
by the ICMAT Severo Ochoa project
SEV-2015-0554 of the Ministry of Economy and Competitiveness
of Spain and the European Regional
Development Fund (FEDER).


We express our gratitude to the referee for many helpful suggestions,
in particular, for improving the estimate in
Theorem~~\ref{z by psi}.


\end{document}